\documentclass[11pt]{amsart}
\usepackage[utf8]{inputenc}
\usepackage{amsmath,amssymb}
\usepackage{wrapfig}
\usepackage{url}
\usepackage{mathtools}
\usepackage{graphicx}
\usepackage{stmaryrd}
\usepackage{amsthm}
\usepackage{xcolor}
\usepackage[colorlinks=true,linkcolor=blue,citecolor=blue]{hyperref}
\usepackage[shortlabels]{enumitem}
\usepackage{comment}
\usepackage{relsize}
\usepackage{dsfont}
\usepackage{stmaryrd}
\usepackage{geometry}
\usepackage{setspace}
\usepackage{mathrsfs} 
\makeatletter
\renewcommand{\author}[2][]{%
  \def\@tempa{#1}
  \ifx\@empty\authors
    \ifx\@tempa\@empty
      \gdef\shortauthors{#2}%
    \else
      \gdef\shortauthors{#1}%
    \fi
    \gdef\authors{\author{#2}}%
  \else
    \ifx\@tempa\@empty
      \g@addto@macro\shortauthors{\and#2}%
    \else
      \g@addto@macro\shortauthors{\and#1}%
    \fi
    \g@addto@macro\authors{\and\author{#2}}%
  \fi
}
\renewcommand{\address}[2][]{\g@addto@macro\authors{\address{#1}{#2}}}
\def\@setauthors{%
  \begin{center}%
    \footnotesize
    \vspace{20pt}
    \let\and\@empty
    \def\author##1{\advance\@tempcnta\@ne}%
    \def\address##1##2{\advance\@tempcntb\@ne}%
    \@tempcnta=\z@  \@tempcntb=\z@
    \authors
    \ifnum\@tempcnta>\@ne \ifnum\@tempcntb=\@ne
        \oneaddress
      \else
        \sepaddresses
      \fi
    \else
      \oneaddress
    \fi
  \end{center}%
}
\def\oneaddress{%
  \begingroup
  \let\author\@iden \let\address\@gobbletwo
  \renewcommand{\andify}{%
    \nxandlist{\unskip, }{\unskip{} and~}{\unskip, and~}}%
  \uppercasenonmath\authors
  \andify\authors
  \authors
  \endgroup
  \begingroup \let\and\relax \let\author\@gobble
  \def\address##1##2{\unskip\\[10pt] \itshape##2}%
  \authors
  \endgroup
}
\def\sepaddresses{%
  \begingroup
    \baselineskip10\p@\relax
    \def\address##1##2{ ({\itshape##2}\/)}
    \def\author##1{\def\temp{##1}\leavevmode\uppercasenonmath\temp\temp}%
    \nxandlist
      {,\\[\baselineskip]}
      {\\[\baselineskip] \textsc{\lowercase{and}}\\[\baselineskip]}
      {,\\[\baselineskip]\textsc{\lowercase{and}}\\[\baselineskip]}
      \authors 
    \authors
  \endgroup
}
\def\maketitle{\par
  \@topnum\z@
  \@setcopyright
  \thispagestyle{firstpage}%
  \uppercasenonmath\shorttitle
  \ifx\@empty\shortauthors \let\shortauthors\shorttitle
  \else
    \newcommand{\@xuppercasenonmath}[1]{\toks@\@emptytoks
      \@xp\@skipmath\@xp\@empty##1$$%
      \edef##1{\@nx\protect\@nx\@upprep\the\toks@}}%
    \@xuppercasenonmath\shortauthors
    \def\@@and{AND}
    \renewcommand{\andify}{%
      \nxandlist{\unskip, }{\unskip{ }\@@and{ }}{\unskip, \@@and{ }}}%
    \andify\shortauthors
  \fi
  \@maketitle@hook
  \begingroup
  \@maketitle
  \endgroup
  \c@footnote\z@
  \@cleartopmattertags
}
\def\@maketitle{%
  \normalfont\normalsize
  \let\@makefntext\noindent
  \@adminfootnotes
  \ifx\@empty\addresses\else \@footnotetext{\@setotheraddresses}\fi
  \global\topskip68\p@\relax
  \@settitle
  \ifx\@empty\authors \else \@setauthors \fi
  \ifx\@empty\@dedicatory
  \else
    \baselineskip26\p@
    \vtop{\centering{\footnotesize\itshape\@dedicatory\@@par}%
      \global\dimen@i\prevdepth}\prevdepth\dimen@i
  \fi
  \toks@\@xp{\shortauthors}\@temptokena\@xp{\shorttitle}%
  \edef\@tempa{\@nx\markboth{\the\toks@}{\the\@temptokena}}\@tempa
  \@setabstract
  \normalsize
  \if@titlepage
    \newpage
  \else
    \dimen@34\p@ \advance\dimen@-\baselineskip
    \vskip\dimen@\relax
  \fi
} 
\renewcommand{\thanks}[1]{%
  \ifx\@empty\thankses
    \gdef\thankses{\thanks{#1}}%
  \else
    \g@addto@macro\thankses{\endgraf\thanks{#1}}%
  \fi}
\def\@setthanks{\def\thanks##1{\noindent##1\@addpunct.}\thankses}
\renewcommand{\curraddr}[2][]{%
  \ifx\@empty\addresses
    \gdef\addresses{\curraddr{#1}{#2}}%
  \else
    \g@addto@macro\addresses{\endgraf\curraddr{#1}{#2}}%
  \fi}
\renewcommand{\email}[2][]{%
  \ifx\@empty\addresses
    \gdef\addresses{\email{#1}{#2}}%
  \else
    \g@addto@macro\addresses{\endgraf\email{#1}{#2}}%
  \fi}
\renewcommand{\urladdr}[2][]{%
  \ifx\@empty\addresses
    \gdef\addresses{\urladdr{#1}{#2}}%
  \else
    \g@addto@macro\addresses{\endgraf\urladdr{#1}{#2}}%
  \fi}
\def\@setotheraddresses{%
  \def\curraddr##1##2{\noindent
    \emph{Current address\@ifnotempty{##1}{ of ##1}}:\space
      ##2\@addpunct.}%
  \def\email##1##2{\noindent
    \emph{E-mail address\@ifnotempty{##1}{ of ##1}}:\space
      \texttt{##2}}%
  \def\urladdr##1##2{\noindent
    \emph{WWW address\@ifnotempty{##1}{ of ##1}}:\space
      \texttt{##2}}%
  \addresses
}
\let\enddoc@text\relax
\makeatother

\newtheorem{theorem}{Theorem}
\numberwithin{theorem}{section} 
\numberwithin{equation}{section}
\newtheorem{corollary}{Corollary}
\numberwithin{corollary}{section} 

\numberwithin{prop}{section} 
\newtheorem{remark}{Remark}
\numberwithin{remark}{section} 
\newtheorem{example}{Example}
\numberwithin{example}{section} 

\newtheorem{lemma}{Lemma}
\numberwithin{lemma}{section}
\theoremstyle{definition}

\theoremstyle{definition}

\numberwithin{definition}{section} 
\DeclareMathOperator{\R}{\mathbb{R}}
\DeclareMathOperator{\Z}{\mathbb{Z}}
\DeclareMathOperator{\N}{\mathbb{N}}

\DeclareMathOperator{\Tr}{Tr}

\title{Global hypoellipticity of $G$-invariant operators on homogeneous vector bundles}

\author[D. Cardona]{Duv\'an Cardona}
\address{
  Department of Mathematics: Analysis, Logic and Discrete Mathematics
  \endgraf
  Ghent University, Belgium
  }
  \email{duvan.cardonasanchez@ugent.be}
\author[A. Kowacs]{André Pedroso Kowacs}
\address{
  Department of Mathematics
  \endgraf
 Universidade Federal do Paraná (UFPR), Brazil
  }
\email{andrekowacs@gmail.com}
\thanks{The authors are supported  by the FWO  Odysseus  1  grant  G.0H94.18N:  Analysis  and  Partial Differential Equations and by the Methusalem programme of the Ghent University Special Research Fund (BOF)
(Grant number 01M01021). Duv\'an Cardona has been supported by the FWO Fellowship
grant No 1204824N of the Belgian Research Foundation FWO. Andr\'e Kowacs was supported in part by the Coordenação de Aperfeiçoamento de Pessoal de N\'ivel Superior - Brasil (CAPES) - Finance Code 001”;}
 


\date{\today}

\begin{document}
\maketitle
\allowdisplaybreaks
\begin{abstract}
We establish necessary and sufficient conditions for the global hypoellipticity of $G$-invariant operators on homogeneous vector bundles. These criteria are established in terms of the corresponding matrix-valued symbols as developed by Ruzhansky and Turunen and extended in \cite{HomoVector} to homogeneous vector-bundles.  
\end{abstract}

\section{Introduction}
\subsection{Outline} The aim of this paper is to apply the matrix-valued quantisation in \cite{HomoVector} on homogeneous vector bundles over compact homogeneous manifolds to the characterisation of hypoellitic $G$-invariant operators, extending the results in \cite{Kirilov_2020}, see Subsection \ref{Subsection:hom} for details.  Since the notion of {\it hypoellipticity} is the main topic of this paper we present this definition as follows. Let $M$ be a compact manifold without boundary. By following H\"ormander \cite{Hormander:1961,Hormander:1967}, a differential operator $D$ with $C^\infty$-coefficients acting on the space of distributions  $\mathscr{D}'(M)$  is called {\it globally hypoelliptic} if the equation 
\begin{equation}
    Du=f,
\end{equation}only has solution $u\in C^\infty(M)$ when $f\in C^\infty(M).$ This definition is easily extended when $D$ is a pseudo-differential operator acting on smooth sections $\Gamma^\infty(E)$ of a vector-bundle $E$ over $M.$ In the case where $M=G/K$ is a compact homogeneous manifold and $E$ is a $G$-vector bundle on $M,$
this work aims to investigate necessary and sufficient conditions for  the global hypoellipticity of $G$-invariant pseudo-differential operators acting on smooth sections of $E.$ In view of the characterisation of hypoelliptic differential operators with constant coefficients, see \cite{Hormander1985III}, where criteria can be obtained in terms of the {\it principal symbol} of the operator, here we adopt a new perspective and our criteria are developed in terms of the (global) {\it matrix-valued symbols} as developed by Ruzhansky and Turunen \cite{Ruz}. The  global hypoellipticity property for pseudo-differential operators  has been widely studied, see for instance \cite{BergamascoGH,BergamascoKirilovGH,PsudoGHKirilov,GF1,GF2,GFremarks,Kirilov_2020,KirilovTS,KirilovCOMP,PseudoTorus}. The main results of this note are presented in Section \ref{Main:results}.

\subsection{State-of-the-art} The hypoellipticity of $G$-invariant pseudo-differential operators $D$ is considered in the setting of $G$-vector bundles. The research about the hypoellipticity of a differential operator with a principal symbol having
real-valued coefficients and constant signs has been a field of intensive research, where typically,  the following two methods are applied. The
first approach is based on an {\it a priori estimate} of solutions in the scale of the Sobolev spaces. Once one finds such
an inequality, one can prove the hypoellipticity of the operator with the aid of interpolation inequalities or with the application of the theory of pseudo-differential operators. The second approach is based on the study of fundamental 
solutions, which is a very well-known difficult problem, in particular if the operator has variable coefficients. For a general analysis about the hypoellicity of linear partial differential operators we refer to H\"ormander \cite{Hormander1985III}. 

As it is well-known, the extensive class of elliptic operators provides fundamental examples of hypoelliptic operators. This fact suggests that the principal
symbol of hypoelliptic operators has a kind of positivity. In 1967, H\"ormander \cite{Hormander:1967} proved the hypoellipticity of second order differential operators 
when 
the principal symbol does not change of sign when the dual variables vary. The following is a natural question that arose after H\"ormander's work \cite{Hormander:1967}

{\it “Is it necessary for hypoellipticity
that the principal symbol does not change sign when the space variables vary?}  

In 1971,
Kannai \cite{Kannai}  proved that $$L_1 = \partial_t + t\partial_{x}^2,$$ acting in smooth functions on $\mathbb{R}^2$ is hypoelliptic, while
$$L_2 =\partial_t - t\partial_{x}^2$$ is not, where the key remark is that 
 $L_1$ and $L_2$ are typical examples of operators with sign-changing
principal symbols. This example illustrates that the semi-definiteness of the principal symbol is
not necessary for hypoellipticity and that the type of changing sign is important. The sign
of the principal symbol of $L_2$ above changes from minus to plus as $t$ increases. This is a
condition similar to the Nirenberg-Treves criterion for local solvability of differential equations
of principal type, see \cite{NirenbergTreves63}. In 1976, Beals and Fefferman \cite{Beals:Fefferman} extended  Kannai’s result to higher dimensional cases. Their result was obtained by obtaining  a suitable  {\it a priori
estimate with weight} and using the general theory of pseudo-differential operators due to
Beals \cite{Beals,Beals2}. An important characteristic of this method is that the hypoellipticity property follows from a
single a priori estimate. However, there are many restrictions on the weight.
Therefore, the class of functions that control the sign of the principal symbol is strictly
limited. Beals and Fefferman’s result was extended by many authors, see e.g. Kumano-go and
Taniguchi \cite{Kumano-go:Taniguchi73} and Lanconelli \cite{Lanconelli}. We also refer to the seminal 1987's work of Morimoto \cite{Morimoto}. As we will show in this note, the hypoellipticity of $G$-invariant pseudo-differential operators can be effectively analysed by applying the {\it matrix-valued quantisation formula} \eqref{vectorquantisation}. Indeed, this is a new point of view in counterpart to the aforementioned results based on the notion of {\it principal symbol}.

 In the setting of a compact Lie group $G$, pseudo-differential operators can be analysed in terms of the corresponding  global matrix-valued symbols as introduced by Ruzhansky and Turunen \cite{Ruz}. Indeed, a pseudo-differential operator $A$ can be writen in terms of the group Fourier transform $\hat{}$ as 
 \begin{equation*}
    Af(x) = \sum_{[\xi]\in\widehat{G}}d_\xi\Tr\left(\xi(x)\sigma_A(x,\xi)\widehat{f}(\xi)\right),
\end{equation*}where the sum runs over the elements of the unitary dual $\widehat{G}$ of $G,$ namely, the set formed by all equivalence classes of continuous irreducible unitary representation of the group $[\xi],$ see Subsection \ref{Preliminaries} for details. Ruzhansky-Turunen's matrix-valued symbols can be extended also in the setting of homogeneous vector bundles, see Section \ref{Preliminaries}. 

While in the standard theory of pseudo-differential operators on manifolds many criteria for the properties of the operator depend on the principal symbol, on compact Lie groups, properties of the operator involve the (full) matrix-valued symbol.  We refer the reader to \cite{CR20} and the extensive list of references therein that include matrix-valued symbol criteria for several properties of pseudo-differential operators: boundedness of pseudo-differential oeprators on $L^p$ spaces, on  $L^p$-Sobolev spaces, on Triebel-Lizorkin spaces, on Besov spaces, etc., as well as applications of the matrix-valued symbols to the analysis of elliptic and parabolic equations, G\r{a}rding type inequalities and global well-posedness of pseudo-differential problems. We remark that as for the analysis of the hypoellipticity of homogeneous differential operators on homogeneous  vector-bundles, the main motivation comes from the index-type theorems in this setting, see the seminal work of Bott \cite{Bott1965} for details. 

\subsection{Structure of the note}
We organise this note as follows. 
In Section \ref{Preliminaries} we introduce the notation used in this work.
In Section \ref{Main:results} we state our main results, present their proofs, and finally provide some examples.

\section{Preliminaries}\label{Preliminaries}

In this section, we introduce the notations  for the development of this work. We first present the basics of representation theory on compact groups, which can be found in more detail in \cite[Chapter 7]{Ruz}.

\subsubsection{The group Fourier transform}
Let $G$ be a compact Lie group. We denote by $\widehat{G}$ its unitary dual. As every $[\xi]\in\widehat{G}$ is finite dimensional, we can always choose a representative that is matrix-valued. As a consequence of the Peter-Weyl Theorem, the collection of the coefficient functions of all such matrices is an orthogonal basis for $L^2(G)$, where integration is taken with respect to the Haar measure in $G$. Then, define the matrix-valued Fourier coefficients of $f\in L^2(G)$ by 
\begin{equation*}
    \widehat{f}(\xi) \doteq \smallint_G f(x)\xi(x)^*dx,
\end{equation*}
for $[\xi]\in\widehat{G}$, and write the Fourier inversion formula as
\begin{equation*}
    f(x) = \sum_{[\xi]\in\widehat{G}}d_\xi\Tr\left(\xi(x)\widehat{f}(\xi)\right),\,\forall x\in G,
\end{equation*}
where $d_\xi\doteq \dim (\xi)$.
In the case where $G$ is a compact Lie group, the coefficient functions of elements of $\widehat{G}$ are smooth, so we can extend this definition to the set of distributions $\mathcal{D}'(G)$ by 
\begin{equation*}
    \widehat{u}(\xi) = \langle u,\xi^*\rangle,
\end{equation*}
for $u\in\mathcal{D}'(G)$ and $[\xi]\in\widehat{G}$, where this evaluation is to be understood coefficient-wise. \\
Now we summarize the quantisation of pseudo-differential operators by matrix-valued symbols, following \cite[Chapter 10]{Ruz}.

\subsubsection{The quantisation formula on compact Lie groups}
Let $G$ now be a compact Lie group. To a continuous linear operator $A:C^\infty(G)\to C^\infty(G)$ we can associate its global matrix-valued symbol $\sigma_A$ by
\begin{equation*}
    \sigma_A(x,\xi) = \xi(x)^*A\xi(x),\,\forall(x,[\xi])\in G\times\widehat{G},
\end{equation*}
so that we can write the quantisation
\begin{equation*}
    Af(x) = \sum_{[\xi]\in\widehat{G}}d_\xi\Tr\left(\xi(x)\sigma_A(x,\xi)\widehat{f}(\xi)\right),\,\forall x\in G.
\end{equation*}
In this setting, left-invariant continuous linear operators are Fourier multipliers characterized by having a matrix-valued symbol which independent on the variable $x\in G$, i.e.: $\sigma_A(x,\xi)=\sigma_A(\xi)$.

\subsubsection{Vector Bundles}
Let $E,X$ be topological spaces, and let $\mathbb{K}$ be a field. Following \cite[Chapter 1]{Wallach}, we say that $E$ is a $\mathbb{K}$-vector bundle over $X$ if there exists a continuous mapping $p:E\to X$ such that
\begin{enumerate}
    \item For every $x\in X$, $E_x\doteq p^{-1}(x)$ is a finite dimensional $\mathbb{K}$-vector space.
    \item For each $x\in X$, there exists a neighbourhood $U$ of $x$ and a homeomorphism $\psi:p^{-1}(U)\to U\times\mathbb{K}^n$ such that $\psi(v)=(p(v),f(v))$, where $f:E_{p(v)}\to \mathbb{K}^n$ is a linear map between vector spaces.
\end{enumerate}
Let $p:E\to X$ be a vector bundle. A continuous map $s:X\to E$ is called a section of $E$ if for all $x\in X$, $p(s(x)) = p(x)$, or equivalently, $s(x)\in E_x$. We denote by $\Gamma (E)$ the set of all sections of $E$. If $X,E$ are smooth manifolds, we also define $\Gamma^\infty(E)$ the set of all {\it{smooth}} sections of $E$. If $X$ is orientable, the space $L^q(E)$, $1\leq q<\infty$, is then defined as the completion of the set of all smooth sections $s\in \Gamma^\infty(E)$ such that
\begin{equation}
\|s\|_{L^q(E)}\doteq\left(\smallint_X\|s(x)\|_{E_x}^q dx\right)^{\frac{1}{q}}<\infty.
\end{equation}
We shall denote by $\mathcal{D}'(E)$ the set of continuous linear functionals over $\Gamma^\infty(E)$, to which we refer as {\it distributions} on $E$. 
\subsubsection{Homogeneous vector bundles}\label{Subsection:hom}
Next, we record some notions about the basic theory of homogeneous vector bundles, following \cite[Chapter 5]{Wallach}. 
Let $G$ be a compact Lie group and $K$ a closed subgroup of $G$. Denote the quotient $M=G/K$ equipped with its natural compact manifold topology. There exists a natural left action of $G$ on $M$ given by $g\cdot hK=ghK$, for every $g,h\in G$. We say that a $\mathbb{C}$-vector bundle (or $\mathbb{R}$-vector bundle) $p:E\to M$ is a homogeneous vector bundle over $M$ if $G$ acts on $E$ on the left and this action satisfies:
\begin{enumerate}
    \item $g\cdot E_x = E_{gx}$, for all $x\in M$, $g\in G$.
    \item The previously induced mappings from $E_x$ to $E_{gx}$ are linear.
\end{enumerate}
There is natural left action of $G$ on $\Gamma(E)$, $G\times \Gamma(E)\to \Gamma(E)$ given by
\begin{equation*}
    (g\cdot s)(x) = g\cdot s(g^{-1}x),
\end{equation*}
for all $x\in X$, $g\in G$.
Consider $p_1:E\to M$, $p_2:F\to M$, $M=G/K$ to be homogeneous vector bundles, where $K<G$ are compact Lie groups. Let $E_0=p_1^{-1}(K)$, $F_0=p_2^{-1}(K)$ be the fibers at the identity coset. For each $\tau\in\text{Hom}(K,\text{End}(E_0))$, there is a natural right action of $K$ on $G\times E_0$ by $(g,v) = (gk,\tau^{-1}(k)v)$. We denote by $G\times_{\tau} E_0$ the quotient $(G\times E_0)/K$ under this action. Similarly, we define $G\times_{\omega} F_0$. One can show that there always exist $\tau\in\text{Hom}(K,\text{End}(E_0))$, $\omega\in\text{Hom}(K,\text{End}(F_0))$ such that 
\begin{equation*}
    E\cong G\times_{\tau} E_0,\qquad F\cong G\times_{\omega} F_0,
\end{equation*}
are isomorphic as homogeneous vector bundles.\\
Now define $C^{\infty}(G,E_0)^{\tau}\subset C^{\infty}(G,E_0)$ by
\begin{equation}
    C^{\infty}(G,E_0)^{\tau}=\{f\in C^\infty(G,E_0)|\forall g\in G,\forall k\in K, f(gk)=\tau(k)^{-1}f(g)\}.
\end{equation}
In a similar way we define $L^q(G,E_0)^{\tau}$ for all $q\geq 1$ and for the homogeneous vector bundle $p_2:F\to M$. 
\begin{remark}
    As in \cite[Chapter 5]{Wallach}, consider the map $\chi_{\tau}:\Gamma^\infty(E)\to C^{\infty}(G,E_0)^{\tau}$, given by
    \begin{equation}
        \chi_\tau(s)(g)\doteq g^{-1}\cdot s(gK)\equiv (g^{-1}\cdot s)(e_GK)
    \end{equation}
    where $e_G\in G$ is the group identity. This mapping extends to a surjective isometry from $L^2(E)$ to $L^2(G,E_0)^{\tau}$, so that we may identify $\Gamma^\infty(E)\cong C^{\infty}(G,E_0)^{\tau}$.
\end{remark}
\noindent Let $\tilde{D}:\Gamma(E)\to\Gamma(F)$ be a linear operator. We say that $\tilde{D}$ is $G$-invariant (or homogeneous) if 
\begin{equation*}
    \tilde{D}(g\cdot s)=g\cdot(\tilde{D}s),
\end{equation*}
for all $s\in \Gamma(E),\, g\in G$.\\
As mentioned in \cite[Page 120]{Wallach}, every continuous linear map $\tilde{A}:\Gamma^\infty(E)\to\Gamma^\infty(F)$ induces a continuous linear map $A: C^{\infty}(G,E_0)^{\tau}\to C^{\infty}(F,F_0)^{\omega}$ by 
\begin{equation}\label{AtildeA}
    A \doteq \chi_\omega\circ \tilde{A}\circ\chi_{\tau}^{-1}.
\end{equation}

\begin{remark}
    In the case where $\tilde{D}:\Gamma^\infty(E)\to\Gamma^\infty(F)$ is a $G$-invariant (homogeneous) continuous linear map, the mapping $D:C^{\infty}(G,E_0)^{\tau}\to C^{\infty}(F,F_0)^{\omega}$ is a vector-valued Fourier multiplier (as explained below).
\end{remark}

\subsubsection{The quantisation formula on homogeneous vector-bundles}
Next we define the vector valued group Fourier transform, following \cite{HomoVector}, in order to obtain a vector valued quantisation formula. Then, we use the previous identifications to obtain a quantisation formula for homogeneous operators. Choose orthonormal basis $B_{E_0}=\{e_{i,E_0}\}_{i=1}^{d_\tau}$ and $B_{F_0}=\{e_{i,F_0}\}_{i=1}^{d_\omega}$ for $E_0$ and $F_0$ respectively. Denoting by $e_{i,E_0}^*(v)\doteq \langle v,e_{i,E_0}\rangle_{E_0}$, $e_{i,F_0}^*(w)\doteq \langle w,e_{i,F_0}\rangle_{F_0}$, for $v\in E_0$, $w\in F_0$, the vector valued $B_{E_0}$-Fourier transform of $f\in C^{\infty}(G,E_0)$ is given by 
\begin{equation*}
    \widehat{f}(i,\xi) \doteq \smallint_G e_{i,E_0}^*(f(x))\xi(x)^*dx,\,\forall[\xi]\in\widehat{G},\,\forall1\leq i\leq \dim(E_0)=d_\tau.
\end{equation*}
As the coefficient functions of elements of $\widehat{G}$ are smooth, we can extend this definition to the set of distributions $\mathcal{D}'(G,E_0)$ by 
\begin{equation*}
    \widehat{u}(i,\xi) = \langle u,\xi^*\otimes e_{i,E_0}\rangle
\end{equation*}
for $u\in\mathcal{D}'(G,E_0)$ and $[\xi]\in\widehat{G}$, where this evaluation is to be understood coefficient-wise.
From the Peter-Weyl Theorem, the inversion formula can be written as
\begin{equation*}
    f(x) = \sum_{i=1}^{d_\tau}\sum_{[\xi]\in\widehat{G}}d_\xi\Tr\left(\xi(x)\widehat{f}(i,\xi)\right)e_{i,E_0},\,\forall x\in G.
\end{equation*} 
Let $A:C^\infty(G,E_0)\to C^\infty(G,F_0)$ be a continuous linear operator. For $1\leq i\leq d_\tau=\dim(E_0)$, $1\leq r\leq d_\omega=\dim(F_0)$, define the matrix symbol
\begin{equation}\label{vectormatrixsymbol}
    \sigma_A(i,r,x,\xi)\doteq \xi(x)^*[e^*_{r,F_0}[A(\xi_{uv}\otimes e_{r,E_0})(x)]]_{u,v=1}^{d_\xi}
\end{equation}
where $x\in G$, $\xi\in\widehat{G}$. The matrix-valued $(B_{E_0},B_{F_0})$-symbol of $A$ is then defined as the mapping $\sigma_A:\{1\leq i\leq d_\tau\}\times\{1\leq i\leq d_\omega\}\times G\times \widehat{G}\to\bigcup \left\{\mathbb{C}^{d_\xi\times d_\xi}|[\xi]\in\widehat{G}\right\}$ given by \eqref{vectormatrixsymbol}.
One can prove that the quantisation formula can then be written as
\begin{align}\label{vectorquantisation}
    Af(x) = \sum_{r=1}^{d_\omega}\sum_{i=1}^{d_\tau}\sum_{[\xi]\in\widehat{G}}d_\xi\Tr\left(\xi(x)\sigma_A(i,r,x,\xi)\widehat{f}(i,\xi)\right)e_{i,F_0}.
\end{align}
Now let $\tilde{A}:\Gamma^\infty(E)\to\Gamma^\infty(F)$ be a continuous linear operator. Then by \eqref{AtildeA} we can define the matrix-valued $(B_{E_0},B_{F_0})$-symbol of $\tilde{A}$ as $\sigma_{\tilde{A}}=\sigma_A$. It follows from $\eqref{vectorquantisation}$ that in this case, the quantisation formula can be written as
\begin{align}\label{bundlequantisation}
     \tilde{A}s(gK) &= \chi_{\omega}^{-1}\left(\sum_{r=1}^{d_\omega}\sum_{i=1}^{d_\tau}\sum_{[\xi]\in\widehat{G}}d_\xi\Tr\left(\xi(x)\sigma_{\tilde{A}}(i,r,g,\xi)\widehat{\chi_\tau s}(i,\xi)\right)e_{i,F_0}\right)\notag\\
     &\equiv\left(g,\sum_{r=1}^{d_\omega}\sum_{i=1}^{d_\tau}\sum_{[\xi]\in\widehat{G}}d_\xi\Tr\left(\xi(x)\sigma_{\tilde{A}}(i,r,g,\xi)\widehat{\chi_\tau s}(i,\xi)\right)e_{i,F_0}\right)\cdot K,
\end{align}
for every $gK\in M=G/K$, and every section $s\in\Gamma^{\infty}(E)$.
\begin{remark}
    Notice that these definitions all depend on the choice of basis $B_{E_0}$ and $B_{F_0}$ of $E_0$ and $F_0$, respectively. In fact, if $B'_{E_0}$
    $B'_{F_0}$
    is another choice of respective basis, then if we denote by $U_E\in End(E_0)$, $U_F\in End(F_0)$ the corresponding unitary change of basis operators, then for every continuous linear operator $A:C^\infty(G,E_0)\to C^\infty(G,F_0)$ we have
    \begin{equation*}
        \sigma_A(i,r,x,\xi)=\sigma_{U_FAU_E^*}(i,r,x,\xi),
    \end{equation*}
    where on the left we consider the $(B_{E_0},B_{F_0})$-symbol of ${A}$ and on the right we consider the $(B_{E_0}',B_{F_0}')$-symbol of $U_FAU^*_E$.
\end{remark}

 \section{Main Results}\label{Main:results}
In this section we present the contributions of this note. We record the following fundamental notations/facts:
\begin{itemize} 
    \item For every $s\in\mathbb{R},$ the Sobolev space $H^s(G)\subset\mathcal{D}'(G)$ is defined by 
\begin{equation*}
    H^s(G)=\{u\in\mathcal{D}'(G)| \|u\|_{H^s(G)}<+\infty\}
\end{equation*}
where 
\begin{equation*}
    \|u\|_{H^s(G)}=\left(\sum_{[\xi]\in\widehat{G}}d_\xi\langle\xi\rangle^{2s}\|\widehat{u}(\xi)\|_{HS}^2\right)^{\frac{1}{2}}.
\end{equation*}
Here, $\langle\xi\rangle = \sqrt{1+\nu_\xi}$ is the common eigenvalue of $(\text{Id}+\mathcal{L}_G)^{\frac{1}{2}}$ corresponding to the coefficient functions of $[\xi]\in\widehat{G}$, where $\mathcal{L}_G$ is the positive Laplace-Beltrami operator on $G$. Also, $\|A\|_{HS}=\sqrt{\Tr(A^*A)}=\sqrt{\sum_{i,j}A_{ij}^2}$ for any square matrix $A$.
\item For any matrix $A$, define $\lambda_{\min}[A]\geq 0$ its smallest {\it singular value}.
\item We shall also use the following well-known characterizations:
\begin{equation}\label{eqSobolevEmbed}
    \mathcal{D}'(G) = \bigcup_{s\in\mathbb{R}}H^s(G),\qquad C^{\infty}(G)=\bigcap_{s\in\mathbb{R}}H^s(G).
\end{equation}
\item The following fact is a consequence of the Weyl eigenvalue counting formula for the Laplacian, see \cite[Chapter 5]{Wallach}:
$$ \sum_{[\xi]\in\widehat{G}} d_\xi^2\langle\xi\rangle^{-2t}<\infty \iff t>\frac{\dim(G)}{2}.$$
\item The Sobolev space $H^s(E)$ (and similarly $H^s(F)$) for any $s\in\mathbb{R}$, is defined by the norm  
\begin{equation}\label{eqsobonorm}
    \|u\|_{H^{s}(E)}^2=\sum_{[\xi]\in\widehat{G}}d_\xi\langle \xi\rangle^{2s} \sum_{i=1}^{d_\tau}\|\widehat{\chi_{\tau}u}(i,\xi)\|_{HS}^2.
\end{equation}
\item Distributions on $E$ and smooth sections are also characterised in terms of Sobolev spaces as follows:
\begin{equation}\label{eqSobolevEmbedVect}
    \mathcal{D}'(E) = \bigcup_{s\in\mathbb{R}}H^s(E),\qquad \Gamma^{\infty}(E)=\bigcap_{s\in\mathbb{R}}H^s(E).
\end{equation}
\end{itemize}
\begin{remark}
    It is possible that for all $u\in\mathcal{D}'(E)$, $\widehat{\chi_\tau u}(i,\xi)=0$ for all $1\leq i\leq d_\tau$, for some $[\xi]\in\widehat{G}$. We denote the set of all $[\xi]\in\widehat{G}$ for which this is not the case by $\widehat{G}(E)$. Hence, \eqref{eqsobonorm} can be rewritten as
\begin{equation}\label{eqsobonorm2}
    \|u\|_{H^{s}(E)}^2=\sum_{[\xi]\in\widehat{G}(E)}d_\xi\langle \xi\rangle^{2s} \sum_{i=1}^{d_\tau}\|\widehat{\chi_{\tau}u}(i,\xi)\|_{HS}^2.
\end{equation}
\end{remark}
\begin{remark} In the case $E=G\times_{\widehat{1}}\mathbb{C}$ is the trivial bundle, $\Gamma^\infty(E)$ coincides with $C^\infty(M)$, and $\chi_{\tau}:C^\infty(M)\to C^\infty(G)^K$ is just the projective lifting
\begin{equation*}
    \chi_{\tau}f(g)\equiv\dot{f}(g)\doteq f(gK),\,\forall g\in G.
\end{equation*}
Also in this case, the set $\widehat{G}(E)$ this coincides with the set of all $[\xi]$ such that $\Pi_M\xi=0$, where $\Pi_M$ is the orthogonal projection of $L^2(G)$ onto $L^2(M)$.     
\end{remark}
\subsection{Statements}
First, we record Theorem \ref{theoghcompact} below which provides a necessary and sufficient condition on the symbol of a Fourier multiplier on a  compact Lie group for the operator to be globally hypoelliptic. 
 This result was obtained in \cite[Theorem 3.3]{Kirilov_2020} for differential operators strongly invariant with respect to an elliptic operator on general compact manifolds, but we will prove it for pseudo-differential operators on compact Lie groups. 
 We also refer to \cite{GFremarks} for the main ideas used in the proof.
\begin{theorem}\label{theoghcompact}
    Let $G$ be a compact Lie group. If $D$ is a left invariant pseudo-differential operator (Fourier multiplier) on $G$, 
then $D$ is globally hypoelliptic if and only if there exists $C>0$, $k\in\mathbb{R}$ such that
    \begin{equation}\label{eq1}
        \lambda_{\min}[\sigma_D(\xi)]\geq C \langle\xi\rangle^{k},
    \end{equation}
    for all but a finite number of $[\xi]\in\widehat{G}$.
\end{theorem}
Next, we extend the previous theorem to the setting of homogeneous operators (Fourier multipliers) on homogeneous vector bundles.
\begin{theorem}\label{theovector}
Let $p:E\to M$, $p:F\to M$, $M=G/K$ be homogeneous vector bundles, where $K<G$ are compact Lie groups, $E = G\times_\tau E_0$, $F = G\times_\omega F_0$. 
Let $D:\Gamma^\infty(E)\to \Gamma^\infty(F)$ be a continuous homogeneous operator (continuous $G$-invariant operator) from the smooth sections of $E$ to the smooth sections of $F$. For every $[\xi]\in \widehat{G}(E)$ 
set
    \begin{equation}
         m_{\xi}(D)^2 = \min\left\{\sum_{r=1}^{d_\omega}
         \left\|\sum_{i=1}^{d_\tau}
         \sigma_{{D}}(i,r,\xi)v(i,\xi)\right\|_{2}^2\bigg|\,v(i,\xi)\in \mathbb{C}^{d_\xi},\,1\leq i\leq d_\tau,\,\sum_{i=1}^{d_\tau}\|v(i,\xi)\|_2^2=1\right\}.
    \end{equation}
Then $D$ is globally hypoelliptic if and only if there exists $C>0$, $k\in\mathbb{R}$ such that
    \begin{equation}\label{eq3}
        m_\xi(D)\geq C \langle\xi\rangle^{k}
    \end{equation}
    for all but a finite number of $[\xi]\in \widehat{G}(E)$.
\end{theorem}
As a consequence, we obtain a similar characterization of globally hypoelliptic Fourier multipliers on compact homogeneous spaces. Here we are going to follow the following notation from Connolly \cite{ThesisConnolly}.  Let $M=G/K$ be a smooth homogeneous manifold, where $K<G$ are compact Lie groups. Set 
\begin{equation}\label{widehatM}
    \widehat{M} = \{[\xi]\in\widehat{G}|\Pi_M(\xi)\not\equiv 0\},
\end{equation}
where $\Pi_M:C^\infty(G)\to C^\infty(M)$ is the projection given by
\begin{equation*}
    (\Pi_M f)(gK) = \smallint_K f(gk)d\mu_G(k).
\end{equation*}
For each $[\xi]\in\widehat{M}$, let $d_\xi^K\in\mathbb{N}$ be the number of non-zero rows in the matrix coefficients of $\Pi_M\xi$, which we can assume all to be ordered first. Now we present the following criterion.

\begin{corollary}\label{corohomo}
Let $D:C^\infty(M)\to C^\infty(M)$ be a Fourier multiplier on $M$,  
and  $\tilde{D}$ its  projective lifting on $G$. 
Then $D$ is globally hypoelliptic if and only if there exists $C>0$, $k\in\mathbb{R}$ such that
    \begin{equation}\label{eq2}
        \lambda_{\min}[[\sigma_{\tilde{D}}(\xi)]_{d_{\xi}\times d_\xi^K}]\geq C \langle\xi\rangle^{k},
    \end{equation}
    where $[A]_{m\times n}$ denotes the upper left $m$ by $n$ block of the matrix $A$,
    for all but a finite number of $[\xi]\in\widehat{M}$.
\end{corollary}
\subsection{Proofs}
Now, we are going to present the proof of the previous results. 
We start with the following lemma:

\begin{lemma}\label{lemmasingvalue}
Let $A\in\mathbb{C}^{m\times n}$ and $B\in\mathbb{C}^{n\times p}$, where $m,n,p\in\N$, be two compatible matrices of complex entries. Then 
\begin{equation*}\label{ineqsingularvalue}
        \|AB\|_{HS}\geq \lambda_{\min}[A]\|B\|_{HS}.
    \end{equation*}
\end{lemma}
\begin{proof}
Since $A^*A$ is a normal square matrix, it follows from the spectral theorem that we can write  $A^*A=Q^*\Lambda Q$, where $\Lambda\in\mathbb{C}^{n\times n}$ is a diagonal matrix whose entries are given by the eigenvalues of $A^*A$, which correspond to the singular values of $A$ squared, and $Q\in\mathbb{C}^{n\times n}$ is unitary. Then 
\begin{align*}
    \|AB\|_{HS}^2&=\Tr(B^*A^*AB)\\
    &=\Tr(B^*Q^*\Lambda QB)\\
    &=\Tr((QB)^*\Lambda QB)\\
    &\geq \lambda_{\min}[A]^2\Tr((QB)^*QB)\\
    &=\lambda_{\min}[A]^2\Tr(B^*B)=\lambda_{\min}[A]^2\|B\|_{HS}^2,
\end{align*}    
which implies the claim.
\end{proof}
\begin{proof}[Proof of Theorem \ref{theoghcompact}]
    Suppose $D$ satisfies \eqref{eq1} for all $[\xi]\in V$, with $\widehat{G}\backslash V$ finite. Then, if $u\in \mathcal{D}'(G)$ is such that $Du\in C^\infty(G)$, we have for $s\in\R$ that
    \begin{align}\label{eqghcompact1}
        +\infty>\|Du\|_{H^s(G)}^2 \geq \sum_{[\xi]\in V}d_\xi\langle\xi\rangle^{2s}\|\widehat{Du}(\xi)\|_{HS}^2=  \sum_{[\xi]\in V}d_\xi\langle\xi\rangle^{2s}\|\sigma_D(\xi)\widehat{u}(\xi)\|_{HS}^2.
        \end{align}
        Applying Lemma \ref{lemmasingvalue} to this inequality yields
    \begin{align*}
        +\infty>\|Du\|_{H^s(G)}^2 &\geq \sum_{[\xi]\in V}d_\xi\langle\xi\rangle^{2s}\lambda_{\min}[\sigma_D(\xi)]^2\|\widehat{u}(\xi)\|_{HS}^2\\
        &\geq C^2\sum_{[\xi]\in V}d_\xi\langle\xi\rangle^{2(s+k)}\|\widehat{u}(\xi)\|_{HS}^2.
    \end{align*}
    Since $\widehat{G}\backslash V$ is finite, we conclude that $\|u\|_{H^{s+k}(G)}<+\infty$ also. Since this holds for any $s\in\mathbb{R}$, we conclude, by \eqref{eqSobolevEmbed}, that $u\in C^\infty(G)$ and thus $D$ is globally hypoelliptic. Suppose now that $D$ does not satisfy \eqref{eq1}. Thus, for each $k\in\mathbb{N}$, there exists distinct $\xi_k\in\widehat{G}$, 
    $v_k\in \mathbb{C}^{d_{\xi_k}}$, $\|v_k\|_2=1$, $\lambda_k\geq 0$ such that
    \begin{equation*}
   \sigma_D(\xi_k)^*\sigma_D(\xi_k)v_k=\lambda_k^2 v_k,
    \end{equation*}
    and
    \begin{equation*}
        \lambda_k<2^{-k}\langle \xi_k\rangle^{-k},
    \end{equation*}
    for every $k\in \N$. Let $u\in \mathcal{D}'(G)$ be defined by the Fourier coefficients
    \begin{align*}
        \widehat{u}(\xi) = 
\begin{bmatrix}
    \vert & \vert &\dots&\vert\\
    v_k   & 0 &\dots&0  \\
    \vert & \vert&\dots&\vert
\end{bmatrix},
    \end{align*}
    if $\xi=\xi_k$, and $ \widehat{u}(\xi) =0$ otherwise.
    Then $u \in \mathcal{D}'(G)\backslash C^\infty(G)$ clearly, as $\|\widehat{u}(\xi_k)\|_{HS}=1$, for all $k\in\N$, but on the other hand
    \begin{align*}
        \|\widehat{Du(\xi_k)}\|_{HS} = \|\sigma_D(\xi_k)\widehat{u}(\xi_k)\|_{HS} = \lambda_k\|\widehat{u}(\xi_k)\|_{HS}<2^{-k}\langle \xi_k\rangle^{-k},
    \end{align*}
    for all $k\in\N$ and $\|\widehat{Du}(\xi)\|_{HS}=0$ for all other $[\xi]\in\widehat{G}$, 
    therefore $Du\in C^\infty(G)$ and so $D$ is not globally hypoelliptic.
\end{proof}
\begin{remark}
    The proof of Theorem \ref{theovector} is similar to the one of Theorem \ref{theoghcompact}, but there are some additional facts to take in consideration for the vector-valued setting, as well as the need to consider the restriction to the set $\widehat{G}(E)$. 
\end{remark}

\begin{proof}[Proof of Theorem \ref{theovector}]
    Indeed, suppose $D$ satisfies the inequality over $V\subset\widehat{G}(E)$. Then for $u\in \mathcal{D}'(E)$ such that $Du\in \Gamma^\infty(E)$, and for any $s\in\mathbb{R}$:
    \begin{align}\label{ineqprooftheovector1}
        +\infty>\|Du\|^2_{H^s(F)} &= \sum_{[\xi]\in\widehat{G}(F)}d_\xi\langle\xi\rangle^{2s} \sum_{r=1}^{d_\omega}\|\widehat{\chi_\omega Du}(r,\xi)\|_{HS}^2\\
        &=\sum_{[\xi]\in\widehat{G}(F)}d_\xi\langle\xi\rangle^{2s}\sum_{r=1}^{d_\omega}\left\|\sum_{i=1}^{d_\tau }\sigma_D(i,r,\xi)\widehat{\chi_\tau u}(i,\xi)\right\|_{HS}^2,\notag
    \end{align}
    where we have used the quantisation formulas \eqref{vectorquantisation} and \eqref{bundlequantisation}.
        But then, 
        \begin{align*}
           \sum_{r=1}^{d_\omega} \left\|\sum_{i=1}^{d_{\tau}}\sigma_{D}(i,r,\xi)\widehat{\chi_\tau u}(i,\xi)\right\|_{HS}^2\geq 
           m_{\xi}(D)^2\sum_{i=1}^{d_\tau}\|\widehat{\chi_\tau u}(i,\xi)\|_{HS}^2.
        \end{align*}
        Indeed, note that for $1\leq r\leq d_\omega$, $\|\sum_{i=1}^{d_{\tau}}\sigma_{D}(i,r,\xi)\widehat{\chi_\tau u}(i,\xi)\|_{HS}^2$ is equal to  $\sum_{n=1}^{d_\xi}\|\sum_{i=1}^{d_\tau}\sigma_D(i,r,\xi)w(i,n,\xi)\|_2^2$, where $w(i,n,\xi)$ is the $n$-th column of $\widehat{\chi_\tau u}(i,\xi)$. But 
        \begin{equation*}
        \sum_{r=1}^{d_\omega}\left\|\sum_{i=1}^{d_\tau}\sigma_D(i,r,\xi)\frac{w(i,n,\xi)}{\sum_{j=1}^{d_\tau}\|w(j,n,\xi)\|_2}\right\|_2^2\geq m_\xi(D)^2,
        \end{equation*}
        by definition, so
        \begin{equation*}
            \sum_{r=1}^{d_\omega} \left\|\sum_{i=1}^{d_{\tau}}\sigma_{D}(i,r,\xi)\widehat{\chi_\tau u}(i,\xi)\right\|_{HS}^2= \sum_{r=1}^{d_\omega}\sum_{n=1}^{d_\xi} \left\|\sum_{i=1}^{d_{\tau}}\sigma_{D}(i,r,\xi){w}(i,n,\xi)\right\|_{2}^2\geq m_\xi(D)^2\sum_{j=1}^{d_\tau}\left\|\widehat{\chi_\tau u}(j,\xi)\right\|_{HS}^2,
        \end{equation*}
        since 
        \begin{equation*}
\sum_{n=1}^{d_\xi} \sum_{j=1}^{d_{\tau}}\left\|{w}(j,n,\xi)\right\|_{2}^2= \sum_{j=1}^{d_\tau}\left\|\widehat{\chi_\tau u}(j,\xi)\right\|_{HS}^2.
        \end{equation*}
        Hence, \eqref{ineqprooftheovector1} implies that
    \begin{align*}    
         +\infty>\|Du\|^2_{H^s(F)}&\geq\sum_{[\xi]\in V} d_\xi\langle\xi\rangle^{2s} m_\xi(D)^2\sum_{i=1}^{d_\tau}\|\widehat{\chi_\tau u}(i,\xi)\|_{HS}^2\\
        & \geq C^2\sum_{[\xi]\in V}d_\xi\langle\xi\rangle^{2(s+k)}\sum_{i=1}^{d_\tau}\|\widehat{\chi_\tau u}(i,\xi)\|_{HS}^2.
    \end{align*}
    Since $\widehat{G}(E)\backslash V$ is finite, we conclude that $\|u\|_{H^{s+k}(E)}<+\infty$ also. Since this holds for any $s\in\mathbb{R}$, we conclude by \eqref{eqSobolevEmbedVect} that $u\in\Gamma^\infty(E)$ and thus $D$ is globally hypoelliptic. 
     Suppose now that $D$ does not satisfy \eqref{eq3}. Then, in particular, for each $k\in\mathbb{N}$, there exists distinct, $\xi_k\in\widehat{G}(E),\,  
    v_k(i)\in \mathbb{C}^{d_{\xi_k}}$, $1\leq i\leq d_\tau$, $\sum_{i=1}^{d_\tau}\|v_k(i)\|^2_2=1$, $\lambda_k
    \geq0$ such that
    \begin{equation*}
    \sum_{r=1}^{d_\omega}\left\|\sum_{i=1}^{d_\tau}
  [\sigma_{{D}}(i,r,\xi_k)]v_k(i)\right\|_{2}^2=\lambda_k^2,
    \end{equation*}
    and
    \begin{equation*}
        \lambda_k^2<2^{-k}\langle \xi_k\rangle^{-k},
    \end{equation*}
    for every $k\in \N$. Let $u\in \mathcal{D}'(E)$ be defined by the Fourier coefficients
    \begin{align*}
        \widehat{\chi_{\tau}u}(i,\xi) = 
\begin{bmatrix}
    \vert & \vert &\dots&\vert\\
    v_k(i)   & 0 &\dots&0  \\
    \vert&\vert&\dots&\vert
\end{bmatrix},
    \end{align*}
    for $i=1,\dots d_\tau$, if $\xi=\xi_k$, and $\widehat{\chi_\tau u}(i,\xi)=0$ otherwise. Then $u\in\mathcal{D}'(E)\backslash \Gamma^\infty(E)$ clearly, as $\sum_{i=1}^{d_\tau}\|\widehat{\chi_\tau u}(i,\xi_k)\|_{HS}^2=1$, for all $k\in\mathbb{N}$, and $\widehat{\chi_\tau u}(i,\xi)=0$ for all other $[\xi]\in \widehat{G}(E)$, and also $\widehat{\chi_\tau u}(i,\xi)=0$, for all $[\xi]\not\in\widehat{G}(E)$, but on the other hand
      \begin{align*}
        \sum_{r=1}^{d_\omega}\|\widehat{\chi_{\tau}{Du}}(r,\xi_k)\|_{HS}^2=\sum_{r=1}^{d_\omega}\left\|\sum_{i=1}^{d_\omega}\sigma_{{D}}(i,r,\xi_k)\widehat{\chi_{\tau}u}(i,\xi_k)\right\|_{HS}^2 = \lambda_k^2<2^{-k}\langle \xi_k\rangle^{-k},
    \end{align*}
    for all $k\in\N$, and $\|\widehat{\chi_{\tau}{Du}}(r,\xi)\|_{HS}=0$ for all other $1\leq r\leq d_\omega$, $[\xi]\in\widehat{G}$, 
    therefore $Du\in \Gamma^\infty(F)$ and so $D$ is not globally hypoelliptic.
\end{proof}
\begin{remark}
    As before, the proof of Corollary \ref{corohomo} is also similar to the proof of \ref{theoghcompact}, but now extra care needs to be taken to discard the frequencies which are not relevant for the lifting of operators on the homogeneous space.
\end{remark}
\begin{proof}[Proof of Corollary \ref{corohomo}]
First, notice that smooth functions on $M$ may be seen as smooth sections on the trivial bundle $p:E= G\times_{\widehat{1}}\mathbb{C}\to M$, and so by the remarks given in Section 2, this result follows from the Theorem \ref{theovector}. More precisely, note that the Sobolev norms for $u\in H^s(M)$ are given by:
\begin{equation*}
    \|u\|_{H^s(M)}^2 = \sum_{[\xi]\in\widehat{M}}{d_\xi}\langle\xi\rangle^s\|\widehat{\dot{u}}(\xi)\|^2_{HS},
\end{equation*}
where $\dot{u}$ is the projective lifting of $u$ on $G$, for $s\in\R$ as $\widehat{G}(E)=\widehat{M}$. Also, for $u\in\mathcal{D}'(M)$, its symbol at $[\xi]\in\widehat{M}$ has possibly only the first $d_\xi^K$ lines different from zero. Similarly, the projective lifting of $D$, $\tilde{D}$, has possibly only the first $d_{\xi}^K$ columns different from $0$. Hence, it follows from Lemma \ref{lemmasingvalue} that
\begin{align*}
     \|\widehat{\dot{Du}}(\xi)\|_{HS}^2 =\|\sigma_{\tilde{D}}(\xi)\widehat{\dot{u}}(\xi)\|_{HS}^2 &=  \|[\sigma_{\tilde{D}}(\xi)\widehat{\dot{u}}(\xi)]_{d_\xi^K\times d_\xi}\|_{HS}^2\\
    &\geq\lambda_{\min}[[\sigma_{\tilde{D}}(\xi)]_{d_{\xi}\times d_\xi^K}]^2\|\widehat{\dot{u}}(\xi)_{d_\xi^K\times d_\xi}\|_{HS}^2.
\end{align*}
 Suppose $D$ satisfies \eqref{eq2} for all $[\xi]\in V$, $\widehat{M}\backslash V$ finite. Then, if $u\in \mathcal{D}'(M)$ is such that $Du\in C^\infty(M)$, we have for $s\in\R$ that
    \begin{align*}
        +\infty>\|Du\|_{H^s(M)}^2 &= \sum_{[\xi]\in\widehat{M}}d_\xi\langle\xi\rangle^{2s}\|\widehat{\dot{Du}}(\xi)\|_{HS}^2\geq\sum_{[\xi]\in V}d_\xi\langle\xi\rangle^{2s}\|\widehat{\dot{Du}}(\xi)\|_{HS}^2\\
        &\geq \sum_{[\xi]\in V}d_\xi\langle\xi\rangle^{2s}\lambda_{\min}[[\sigma_{\tilde{D}}(\xi)]_{d_{\xi}\times d_\xi^K}]^2\|\widehat{\dot{u}}(\xi)\|_{HS}^2\\
        &\geq C^2\sum_{[\xi]\in V}d_\xi\langle\xi\rangle^{2(s+k)}\|\widehat{\dot{u}}(\xi)\|_{HS}^2.
    \end{align*}
    Since $\widehat{M}\backslash V$ is finite, we conclude that $\|u\|_{H^{s+k}(M)}<+\infty$ also. Since this holds for any $s\in\mathbb{R}$, we conclude by \eqref{eqSobolevEmbedVect} that $u\in C^\infty(M)$ and thus $D$ is globally hypoelliptic. Suppose now that $D$ does not satisfy \eqref{eq2}. Then, for each $k\in\mathbb{N}$, there exist distinct $\xi_k\in\widehat{M}$, 
    $v_k\in \mathbb{C}^{d_{\xi_k}^K}$, $\|v_k\|_2=1$, $\lambda_k\geq0$ such that
    \begin{equation*}
  [\sigma_{\tilde{D}}(\xi)]_{d_\xi^K\times d_\xi}^*[\sigma_{\tilde{D}}(\xi)]_{d_{\xi}\times d_\xi^K}v_k=\lambda_k^2v_k,
    \end{equation*}
    and
    \begin{equation*}
        \lambda_k<2^{-k}\langle \xi_k\rangle^{-k}.
    \end{equation*}
  Let $u\in \mathcal{D}'(M)$ be defined by the Fourier coefficients
    \begin{align*}
        \widehat{\dot{u}}(\xi) = 
\begin{bmatrix}
    \vert & \vert &\dots&\vert\\
    v_k   & 0 &\dots&0  \\
    \vert & \vert&\dots&\vert\\
    0_{d_\xi-d_\xi^K}&0_{d_\xi-d_\xi^K}&\dots&0_{d_\xi-d_\xi^K}
\end{bmatrix},
    \end{align*}
    if $\xi=\xi_k$, and $ \widehat{u}(\xi) =0$ otherwise.
    Then $u \in \mathcal{D}'(M)\backslash C^\infty(M)$ clearly, as $\|\widehat{\dot{u}}(\xi_k)\|_{HS}=1$, for all $k\in\N$, also, $\widehat{u}(\xi)=0$ for $[\xi]\not\in\widehat{M}$ and only the first $d_{\xi}^K$ lines of $\widehat{u}(\xi)$ are possibly different from zero. But on the other hand
    \begin{align*}
        \|\widehat{\dot{Du}}(\xi_k)\|_{HS} = \|\sigma_{\tilde{D}}(\xi_k)\widehat{\dot{u}}(\xi_k)\|_{HS} = \lambda_k\|\widehat{\dot{u}}(\xi_k)\|_{HS}<2^{-k}\langle \xi_k\rangle^{-k},
    \end{align*}
    for all $k\in\N$ and $\|\widehat{\dot{Du}}(\xi)\|_{HS}=0$ for all other $[\xi]\in\widehat{M}$, 
    therefore $Du\in C^\infty(M)$ and so $D$ is not globally hypoelliptic.
\end{proof}
\subsection{Remarks and examples}
We illustrate our main results in the following examples:
\begin{example}\label{examplediag}
    Notice that if $D$ is a left-invariant pseudo-differential operator on a compact Lie group $G$ with diagonal symbol, e.g. $D$ is given by $cX+d$, where $X$ is a left-invariant vector field on $G$, $c,d\in\mathbb{C}$, the inequality in \eqref{eq1} of Theorem \ref{theoghcompact} can be rewritten as
\begin{equation*}
    |\sigma_D(\xi)_{jj}|\geq C\langle\xi\rangle^k
\end{equation*}
for some $C>0$, $k\in\mathbb{R}$ and all but finitely many $1\leq j\leq d_{\xi},\,[\xi]\in\widehat{G}$. 
\end{example}

\begin{example}
    Let $c\in\mathbb{C}$ and consider $D = \partial_0+c$, a perturbation of the neutral operator $\partial_0$ on the compact Lie group $SU(2)$. Under the usual identification $\widehat{SU(2)}\sim \frac{1}{2}\N_0$, its symbol is given by:
    \begin{equation*}
       \sigma_{\partial_0+c}(\ell)_{ij}=(c+i)\delta_{ij}
    \end{equation*}
    where $-\ell\leq i,j\leq \ell,\,\ell-i,\ell-j\in \mathbb{Z}$, 
    and so if $c\in\frac{1}{2}\Z$ we have that, 
    \begin{equation*}
        \lambda_{\min}[\sigma_{\partial_0+c}] = 0
    \end{equation*}
    for infinitely many $\ell\in\frac{1}{2}\N_0$ and so it does not satisfy inequality \eqref{eq1}. In this case, Theorem \ref{theoghcompact} implies that $\partial_0+c$ is not globally hypoelliptic. On the other hand, if $c\not\in\frac{1}{2}\Z$, it is possible to show that
    \begin{equation*}
         |\sigma_D(\xi)_{jj}|\geq \epsilon
    \end{equation*}
    for some $\epsilon>0$, therefore by the same theorem and Example \ref{examplediag} we conclude that in this case $\partial_0+c$ is globally hypoelliptic.
\end{example}

\begin{example}
    Let $G$ be a compact Lie group and $\mathcal{L}$ a sub-Laplacian satisfying the H\"ormander condition with step $\kappa$. It is negative definite formally self-adjoint, and so by choosing an appropriate basis for the representation spaces of $\widehat{G}$, its symbol is diagonal and given by 
    \begin{equation*}
        \sigma_{\mathcal{L}}(\xi)=\text{diag}\left(-\nu_1^2,\dots,-\nu_{d_\xi}^2\right),
    \end{equation*}
    for some $\nu_{j}(\xi)\geq 0$ and for all $[\xi]\in\widehat{G}$. It follows from the functional calculus, that for any $s\in\mathbb{R}$ the symbol of the operator $(\text{Id}-\mathcal{L})^{\frac{s}{2}}$ is given by
    \begin{equation*}
        \sigma_{(\text{Id}-\mathcal{L})^{\frac{s}{2}}}(\xi)=\text{diag}\left((1+\nu_1^2)^{\frac{s}{2}},\dots,(1+\nu_{d_\xi}^2)^{\frac{s}{2}}\right),
    \end{equation*}
    for all $[\xi]\in\widehat{G}$.
   As proved in Proposition 3.1 of \cite{Garetto_2015} there exist $C_1,C_2>0$ such that
    \begin{equation*}
       C_1\langle\xi\rangle\geq (1+\nu_{j}(\xi)^2)^{\frac{1}{2}}\geq C_2\langle\xi\rangle^{\frac{1}{\kappa}},
    \end{equation*}
    for all $1\leq j\leq d_\xi$, $[\xi]\in\widehat{G}$. It follows from Theorem \ref{theoghcompact} and Example \ref{examplediag} that the operator $(\text{Id}-\mathcal{L})^{\frac{s}{2}}$ is globally hypoelliptic for any $s\in\mathbb{R}$.
\end{example}

\begin{example}
    Let $G$ be a compact Lie group, $\mathcal{L}$ a subelliptic operator such that
    \begin{equation}\label{subellineq}
        \|u\|_{H^{\frac{2}{r}}(G)}\leq C\left(\|\mathcal{L}u\|_{L^2(G)}+\|u\|_{L^2(G)}\right).
    \end{equation}
    Theorem \ref{theoghcompact} provides an alternative proof of the fact that $\mathcal{L}$ is globally hypoelliptic. Indeed, given $[\xi]\in\widehat{G}$, let $v_\xi\in \mathbb{C}^{d_\xi}$ be a unit vector such that 
    \begin{equation}
        \sigma_{\mathcal{L}}(\xi)^*\sigma_{\mathcal{L}}(\xi)v_\xi=\lambda_{\min}[\sigma_{\mathcal{L}}(\xi)]^2v_\xi.
    \end{equation}
    Let $u\in\mathcal{D}'(G)$ be defined by the Fourier coefficients
    \begin{equation*}
 \widehat{u}(\xi)=\begin{bmatrix}
    \vert & \vert &\dots&\vert\\
    v_\xi   & 0 &\dots&0  \\
    \vert&\vert&\dots&\vert
\end{bmatrix},
 \end{equation*}
 and $\widehat{u}(\xi')=  0$ for all other $[\xi']\in\widehat{G}$.
 Then by \eqref{subellineq} and Plancherel's Theorem 
    \begin{align*}
        d_\xi\langle \xi\rangle^{\frac{2}{r}}&\leq C(d_\xi\|\sigma_{\mathcal{L}}(\xi)\widehat{u}(\xi)\|_{HS}^2+d_\xi)\\
        &\leq C(d_\xi
        \lambda_{\min}[\sigma_{\mathcal{L}}(\xi)]^2+d_\xi)
    \end{align*}
    which implies 
    \begin{equation*}
         \lambda_{\min}[\sigma_{\mathcal{L}}(\xi)]\geq \sqrt{\frac{1}{C}\langle\xi\rangle^{\frac{2}{r}}-1}.
    \end{equation*}
    Since $[\xi]\in\widehat{G}$ is arbitrary, we conclude this holds for every $[\xi]\in\widehat{G}$. Therefore, for all but finitely many $[\xi]\in\widehat{G}$, we have the estimate
    \begin{equation*}
        \lambda_{\min}[\sigma_{\mathcal{L}}(\xi)]\gg \langle\xi\rangle^{\frac{1}{r}},
    \end{equation*}
    and so by Theorem \ref{theoghcompact} we conclude $\mathcal{L}$ is globally hypoelliptic. Note that a typical example of an operator $\mathcal{L}$ satisfying the properties above is an arbitrary H\"ormander sub-Laplacian \cite{Hormander:1967}. In this case we refer to Rothschild and Stein \cite{RothschildStein76}  for the validity of the subelliptic estimate  \eqref{subellineq}.
\end{example}
\begin{example}
    Consider the compact homogeneous space $M = \mathbb{S}^2 = SU(2)/\mathbb{T}^1$, where $K=\mathbb{T}^1$ is the (maximal torus) subgroup of $SU(2)$ given in Euler angles by

    \begin{equation*}
        \left\{\begin{pmatrix}
            e^{i\psi/2}&ie^{-i\psi/2}\\ie^{i\psi/2}&e^{-i\psi/2}
        \end{pmatrix}:\psi\in[-2\pi,2\pi)\right\}.
    \end{equation*}
    In this case, note that the only $K$ invariant elements of $\widehat{SU(2)}$ are given by $\mathfrak{t}^\ell_{n0}$, $\ell\in\N_0$, and indeed $\Pi_{\mathbb{S}^2}\mathfrak{t}^\ell\equiv0$ (i.e.:$\Pi_{\mathbb{S}^2}\mathfrak{t}^\ell_{nm}\equiv0,$ for all $-\ell\leq m,n\leq \ell$) if $\ell\in\frac{1}{2}\mathbb{N}_0\backslash\mathbb{N}_0$. Therefore, when $M=\mathbb{S}^2$ in \eqref{widehatM}, $\widehat{\mathbb{S}^2}$ can be parametrized by $\N_0$ and the Fourier inversion formula can be written (in the usual notation) as
    \begin{equation*}
        f(\phi,\theta) = \sum_{\ell\in\N_0}(2\ell+1)\sum_{-\ell\leq n \leq\ell}\mathfrak{t}^\ell_{n0}(\phi,\theta,0)\widehat{\dot{f}}(\ell)_{0n},
    \end{equation*}
    where $\dot{f}(\phi,\theta,\psi) = f(\phi,\theta)$ for $f\in C^\infty(\mathbb{S}^2)$,  $0\leq\phi<2\pi,\,0\leq\theta\leq \pi$ and the sum ranges over $n-\ell\in\mathbb{Z}$.\\
    Hence, by Corollary \ref{corohomo}, a homogeneous operator $D$ on $\mathbb{S}^2$ is globally hypoelliptic if and only if it satisfies 
    \begin{equation*}
        \lambda_{\min}[\text{middle column of }\sigma_{\tilde{D}}(\xi)]=\sqrt{\sum_{-\ell\leq n\leq \ell}|\sigma_{\tilde{D}}(\ell)_{n0}|^2}\geq C (1+\ell)^{k},
    \end{equation*}
    for some $C>0$, $k\in\mathbb{R}$ and all but finitely many $\ell\in\N_0$.

    Taking for instance the Laplacian on the sphere $\Delta_{\mathbb{S}^2}$, then its symbol, in this notation, is given by
    \begin{equation*}
        \sigma_{\tilde{\Delta_{{\mathbb{S}^2}}}}(\ell) =\text{diag}\big(0,\dots,0,\ell^2+\ell,0,\dots,0\big)
    \end{equation*}
    so the inequality above can be re-written as 
    \begin{equation*}
        \ell^2+\ell\geq C(1+\ell)^{k},
    \end{equation*}
    which is indeed true for $C=k=1$, $\forall\ell\in\N_0,\,\ell \neq 0$.

\end{example}
\subsection{Acknowledgement} The authors would like to thank Alexandre Kirilov and Wagner A. Almeida de Moraes for discussions. The results below were presented by the second author at the {\it Ghent Analysis and PDE seminar} and were announced in \cite{CardonaKowaczNote}.
\subsection*{Conflict of interests statement.}   On behalf of all authors, the corresponding author states that there is no conflict of interest.

\subsection*{Data Availability Statements.}  Data sharing not applicable to this article as no datasets were generated or
analysed during the current study.

\end{document}